\documentclass[hidelinks]{article}
\usepackage{graphicx} 
\usepackage[a4paper, margin=2.5cm]{geometry}
\usepackage{hyperref}
\usepackage{amsthm}
\usepackage{thm-restate}
\usepackage{geometry}
\usepackage{amssymb}
\usepackage{amsmath}
\usepackage{cleveref}
\usepackage{comment}
\usepackage{xcolor}
\usepackage[shortlabels]{enumitem}
\usepackage{listings}
\usepackage{blkarray}
\usepackage{lmodern}
\usepackage[english]{babel}
\usepackage{float}
\usepackage{tikz}
\usetikzlibrary{arrows.meta,arrows}
\usetikzlibrary{arrows,decorations.markings}
\usetikzlibrary{decorations.pathmorphing}

\newtheorem{theorem}{Theorem}[section]
\newtheorem{lemma}[theorem]{Lemma}
\newtheorem{corollary}[theorem]{Corollary}

\newtheorem{claim}[theorem]{Claim}

\newtheorem{problem}[theorem]{Problem}

\theoremstyle{definition}
\newtheorem{definition}[theorem]{Definition}

\newcommand{\eps}{\varepsilon}
\newcommand{\Gnp}{G(n,p)}

\newcommand{\Nod}{N_\mathrm{odd}}
\newcommand{\Nev}{N_\mathrm{even}}

\title{The Hamilton space of pseudorandom graphs}
\author{
Micha Christoph\thanks{Department of Computer Science, ETH Z\"urich, Switzerland.
			Email: \href{mailto:micha.christoph@inf.ethz.ch}{\nolinkurl{micha.christoph@inf.ethz.ch}}. This author was supported by SNSF Ambizione grant No. 216071.}
\and
Rajko Nenadov\thanks{School of Computer Science, University of Auckland, New Zealand. Email: \href{mailto:rajko.nenadov@auckland.ac.nz}{\nolinkurl{rajko.nenadov@auckland.ac.nz}}}
\and 
Kalina Petrova\thanks{Department of Computer Science, ETH Z\"urich, Switzerland.
			Email: \href{mailto:kalina.petrova@inf.ethz.ch}{\nolinkurl{kalina.petrova@inf.ethz.ch}}. This author was supported by grant no. CRSII5 173721 of the Swiss National Science Foundation.}}
\date{}

\begin{document}

\maketitle

\begin{abstract}
We show that if $n$ is odd and $p \ge C \log n / n$, then with high probability Hamilton cycles in $\Gnp$ span its cycle space. More generally, we show this holds for a class of graphs satisfying certain natural pseudorandom properties. The proof is based on a novel idea of parity-switchers, which can be thought of as analogues of absorbers in the context of cycle spaces. As another application of our method, we show that Hamilton cycles in a near-Dirac graph $G$, that is, a graph $G$ with odd $n$ vertices and minimum degree $n/2 + C$ for sufficiently large constant $C$, span its cycle space.
\end{abstract}

\section{Introduction}

Given a graph $G = (V, E)$, let $\mathcal{E}(G)$ denote the vector space of subsets of $E$ over $\mathbb{F}_2$, also known as the \emph{edge space}. The elements of this vector space correspond to subsets of $E$, with addition defined as the symmetric difference. The \emph{cycle space} $\mathcal{C}(G) \subseteq \mathcal{E}(G)$ is defined as the subspace spanned by all (not necessarily induced) cycles in $G$, and $\mathcal{C}_k(G) \subseteq \mathcal{C}(G)$ --- by cycles of length exactly $k$. Determining conditions under which $\mathcal{C}_k(G) = \mathcal{C}(G)$, for some $k$, is a well-studied problem in graph theory \cite{kahncyclespace,bondy81cycle,demarco13triangles,hartman83long,locke852conn,locke85basis}.

Here we are interested in the case where $\mathcal{C}_n(G) = \mathcal{C}(G)$, that is, in the case where Hamilton cycles span the cycle space. Note that if $G$ contains an odd cycle, a necessary condition for $\mathcal{C}_n(G) = \mathcal{C}(G)$ is that $G$ has an odd number of vertices. Alspach, Locke, and Witte \cite{alspach90hamilton} showed that if $\Gamma$ is a finite abelian group of odd order, then $\mathcal{C}_n(G) = \mathcal{C}(G)$ for any connected Cayley graph $G$ over $\Gamma$. Heinig \cite{heinig2014prisms} showed that the same holds if $G$ is a graph with odd $n$ vertices and minimum degree $\delta(G) \ge (1 + \eps)n/2$, for any $\eps > 0$ and $n$ sufficiently large, and in \cite{heinighamiltonspace} he obtained  $\mathcal{C}_n(G_{n,p}) = \mathcal{C}(G_{n,p})$ with high probability for $p\geq n^{-1/2+o(1)}$. We improve both of these results. In the case of dense graphs, we improve the minimum degree to the one which is optimal up to an additive constant term.

\begin{theorem} \label{thm:dirac}
    There exists $C > 1$ such that the following holds for sufficiently large odd $n$. Let $G$ be a graph with $n$ vertices such that $\delta(G) \ge n/2 + C$. Then $\mathcal{C}_n(G) = \mathcal{C}(G)$.
\end{theorem}

In the case of random graphs, we improve the bound on $p$ to $p \ge C \log n / n$, which is best possible up to the (multiplicative) constant $C$. More generally, we show an analogous statement for a class of pseudorandom graphs.

\begin{definition}[\cite{THOMASON1987307}]
A graph $G$ is \emph{$(p,\beta)$-jumbled} with $0 < p < 1 \leq \beta$ if for every $U \subseteq V(G)$, it holds that
$$ \Big|e_G(U) - p\binom{|U|}{2} \Big| \leq \beta |U|.$$
\end{definition}

\begin{theorem}
    \label{thm:main_technical}
    There exist $\eps > 0$ and $C > 1$ such that the following holds for sufficiently large odd $n$. Let $G$ be a graph on $n$ vertices, satisfying the following three conditions for some $p \geq C \log{n}/n$:
    \begin{itemize}
        \item $\delta(G)\geq (1-\eps)np$,
        \item $G$ is $(p,\beta)$-jumbled for some $\beta\leq \eps np / \log\log n$,
        \item there exists an edge between every two disjoint subsets of vertices $A, B \subseteq V(G)$ such that 
        $$
            |A|, |B| \ge \eps n \frac{\log{\log{n}}}{\log n }.
        $$
    \end{itemize}
    Then $\mathcal{C}_n(G) = \mathcal{C}(G)$.
\end{theorem}

To the best of our knowledge, it is conceivable that Theorem \ref{thm:main_technical} holds assuming only the first two conditions, and moreover weakening the second condition to $\beta \le \eps np$. However, even showing that such $G$ contains a Hamilton cycle is a major open problem \cite{glock2023hamilton,krivelevich03ham}. On the positive side, we can remove the third assumption if we strengthen the jumbledness property.

\begin{corollary}
\label{cor:jumbled}
There exist $\eps > 0$ and $C > 1$ such that the following holds for sufficiently large $n$. Let $p \geq C\log{n}/n$ and $\beta \leq \eps np / \log{n}$. Then, any $(p,\beta)$-jumbled graph $G$ on $n$ vertices with minimum degree $\delta(G) \geq (1-\eps)np$ satisfies $\mathcal{C}_n(G) = \mathcal{C}(G)$.
\end{corollary}
\begin{proof}
Suppose there exist disjoint $A,B\subseteq V(G)$ of size $k$ each, for which $e_G(A,B)=0$. From $(p, \beta)$-jumbledness we get
$$
    e_G(A\cup B)\geq p k(2k-1)-2\beta k = 2pk(k-1)+pk - 2\beta k
$$
and
$$
    e_G(A),e_G(B)\leq pk(k-1)/2+\beta k.
$$
From $e_G(A \cup B) = e_G(A, B) + e_G(A) + e_G(B)$, we conclude
$$
    pk\leq 4\beta.
$$
Therefore, between any pair of sets of size 
$$
    4\beta/p +1\leq \frac{4\eps n}{\log{n}}+1 \leq \frac{\eps n \log{\log{n}}}{\log{n}}
$$
there is an edge, thus $G$ satisfies all the conditions of Theorem~\ref{thm:main_technical}. 
\end{proof}




Corollary~\ref{cor:jumbled} can readily be applied to $(n,d,\lambda)$-graphs, $d$-regular graphs with $n$ vertices and the second largest absolute value of the adjacency matrix at most $\lambda$, with $d\geq C\log{n}$ and $\lambda\leq  \eps d / \log{n}$. Indeed, as shown in \cite[Lemma 2.3]{ndlambda}, $(n, d, \lambda)$-graphs are $(d/n,\lambda)$-jumbled.

\begin{corollary}
    There exist $\eps > 0$ and $C > 1$ such that the following holds for sufficiently large odd $n$. Every $(n,d,\lambda)$-graph $G$ with $d \geq C \log{n}$ and $\lambda \leq \varepsilon d / \log{n}$ satisfies $\mathcal{C}_n(G) = \mathcal{C}(G)$. 
\end{corollary}

Finally, a simple application of the Chernoff inequality and a union bound shows that Theorem \ref{thm:main_technical} applies to random graphs. The following result was our main motivation for this work.

\begin{corollary} \label{cor:gnp}
There exists $C > 1$ such that for $p \geq C\log{n} / n$, $G \sim \Gnp$ w.h.p.\  satisfies $\mathcal{C}_n(G) = \mathcal{C}(G)$ if $n$ is odd.
\end{corollary}
\begin{proof}
Let us show that $G$ satisfies the three conditions of Theorem~\ref{thm:main_technical}, where we take $C$ large enough compared to $C_{\ref{thm:main_technical}},\eps_{\ref{thm:main_technical}}$ satisfying Theorem~\ref{thm:main_technical}. The first condition follows w.h.p. from a standard application of the Chernoff inequality. It is known, see for example \cite{thomason1987pseudo}, that $G_{n,p}$ is w.h.p. $(p,2\sqrt{np})$-jumbled for $p$ as stated, thus
$$
    2\sqrt{np} \leq \frac{\eps_{\ref{thm:main_technical}} np}{\log\log{n}}
$$
implies the second condition. Finally, let $A,B\subseteq V(G)$ be two disjoint sets of size $\eps_{\ref{thm:main_technical}} n \frac{\log\log n}{\log n}$ each. It follows that 
$$
\Pr[e_G(A,B) = 0] = (1-p)^{|A||B|}\leq e^{-p|A||B|}\leq e^{-\eps_{\ref{thm:main_technical}}^2 C n(\log\log n)^2 / \log n }.
$$
Since there are at most 
$$
    \binom{n}{|A|}\binom{n}{|B|}\leq e^{2\eps_{\ref{thm:main_technical}} n(\log\log n)^2 / \log n}
$$
choices of $A,B\subseteq V(G)$, a union bound yields that w.h.p.\ $G$ also satisfies the third condition. Thus, ensuring $C\geq C_{\ref{thm:main_technical}}$ and $C> 2/\eps_{\ref{thm:main_technical}}$, the corollary follows from Theorem~\ref{thm:main_technical}.
\end{proof}

The rest of the paper is organised as follows. In Section \ref{sec:outline} we lay out the framework and prove the Cycle Lemma (Lemma \ref{lemma:cycle_lemma}), which is the heart of the proof. We then illustrate an application of the framework by giving a simple proof of Theorem \ref{thm:dirac}. Section \ref{sec:prelim} collects known results and properties of pseudorandom graphs. We then give the proof of Theorem \ref{thm:main_technical} in Section~\ref{sec:proof_technical}, and state open problems in Section \ref{sec:concluding}.

\paragraph{Notation.} It is important to keep in mind that throughout the paper we identify a graph with its set of edges. Our notation is standard. Given a graph $G$ and disjoint subsets $S,A \subseteq V(G)$, we use $N_G(S)$ to denote the set of neighbours of vertices from $S$ in $V(G) \setminus S$, and $N_G(S, A) = N_G(S) \cap A$. For disjoint subsets $A, B \subseteq V(G)$, we denote with $G[A, B]$ the induced bipartite subgraph and with $e_G(A, B)$ the number of edges in $G[A, B]$. For a given subgraph $F \subseteq G$, we denote with $V(F) \subseteq V(G)$ the set of vertices incident to edges in $F$.

\section{Proof Strategy}\label{sec:outline}

The following lemma is our starting point. A similar statement is also used as a starting point in \cite{kahncyclespace,demarco13triangles}, however after it the two approaches diverge.

\begin{lemma} \label{lemma:E_r}
  Let $G$ be a Hamiltonian graph with odd $n$ vertices, such that $\mathcal{C}_n(G) \neq \mathcal{C}(G)$. Then there exists a subgraph $R \subseteq G$ for which the following holds:
  \begin{enumerate}[(C1)]
    \item \label{prop:counter1} $R \neq G$,
    \item \label{prop:counter2} every Hamilton cycle in $G$ contains an even number of edges from $R$, and
    \item \label{prop:counter3} for every partition $V(G) = A \cup B$ we have 
    $$
        e_{R}(A, B) \ge e_G(A, B)/2
    $$
    and $R \neq G[A, B]$.
  \end{enumerate}
\end{lemma}
\begin{proof}

Given a subspace $\mathcal{S} \subseteq \mathcal{E}(G)$, we write $\mathcal{S}^\perp$ for its orthogonal complement,
$$
    \mathcal{S}^\perp = \{D \in \mathcal{E}(G) \colon \langle D, S \rangle = 0 \text{ for all } S \in \mathcal{S} \},
$$
where
$$
    \langle D, S \rangle = \sum_{e \in E(G)} D(e) S(e) 
$$
with addition being done in $\mathbb{F}_2$ and, for any $X \in \mathcal{E}(G)$, $X(e)$ is the indicator for $e \in X$. It is well known that $\mathcal{C}^\perp(G)$ is precisely the \emph{cut space} of $G$, that is, it consists of all edge subsets corresponding to induced bipartite graphs $G[A, B]$, including $\varnothing$, for all partitions $V(G) = A \cup B$.

From $\mathcal{C}_n(G) \subseteq \mathcal{C}(G)$ and the assumption of the lemma, we conclude $\mathcal{C}_n^\perp(G)\setminus \mathcal{C}^\perp(G)$ is non-empty. Let $R \subseteq G$ be a largest (in terms of edges) element of $\mathcal{C}_n^\perp(G)\setminus \mathcal{C}^\perp(G)$.  Since $R \in \mathcal{C}_n^{\perp}(G)$, every Hamilton cycle in $G$ has an even number of edges in $R$ (by the definition). Note that $R + \mathcal{C}^\perp(G) \subseteq \mathcal{C}_n^\perp(G)$ is disjoint from $\mathcal{C}^\perp(G)$ since it is a coset of $\mathcal{C}^\perp(G)$.  Therefore, $R$ is not a cut of $G$ but contains at least half the edges over every cut $(A, B)$, as otherwise $R +  G[A, B] \in \mathcal{C}_n^\perp(G) \setminus \mathcal{C}^\perp(G)$ contradicts the maximality of $R$. Finally, we have $R \neq G$, as otherwise $R$ contains a Hamilton cycle, contradicting the fact that every Hamilton cycle has an even number of edges in $R$. It is worth noting that this is the only place where we use the fact that $n$ is odd. 
\end{proof}




Central to our proof is the notion of a \emph{parity-switcher}.
\begin{definition}[Parity-switcher] \label{def: switcher}
    Given a graph $G$ and a subgraph $R \subseteq G$, a subgraph $W \subseteq G$ is called an \textit{$R$-parity-switcher} if it consists of an even cycle $C=(v_1,v_2,\dots,v_{2k})$ with an odd number of edges in $R$, and vertex-disjoint paths $P_i$ between $v_i$ and $v_{2k-i+2}$ for $2\leq i\leq k$.
\end{definition}

A parity-switcher contains two Hamilton paths (that is, paths containing each vertex in $V(W)$):
\begin{itemize}
    \item $v_1 \to v_2 \stackrel{P_2}{\to} v_{2k} \to v_{2k-1} \stackrel{P_3}{\to} v_3 \to v_4 \stackrel{P_4}{\to} v_{2k-2} \to \ldots \to v_{k+1}$, and 
    \item $v_1 \to v_{2k} \stackrel{P_2}{\to} v_{2} \to v_{3} \stackrel{P_3}{\to} v_{2k-1} \to v_{2k-2} \stackrel{P_4}{\to} v_4 \to \ldots \to v_{k+1}$.
\end{itemize}
Importantly, the two paths have a different parity of edges from $R$, since each $P_i$ for $2\leq i\leq k$ appears fully in both paths and each edge of $C$ is contained in exactly one of the paths (see Figure \ref{fig:switcher}).

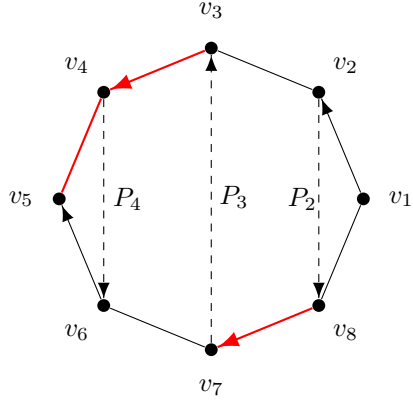
\begin{figure}[h!]   
  \centering
  \begin{tikzpicture}[scale = 1]
    \tikzstyle{blob} = [fill=black,circle,inner sep=1.7pt,minimum size=0.5pt]
    \tikzset{myptr/.style={decoration={markings,mark=at position 1 with %
        {\arrow[scale=1.5,>=latex]{>}}},postaction={decorate}}}
    
    \node[blob] (v1) at (0:2) {}; \node at (0:2.5) {$v_1$}; 
    \node[blob] (v2) at (45:2) {}; \node at (45:2.5) {$v_2$}; 
    \node[blob] (v3) at (45*2:2) {}; \node at (45*2:2.5) {$v_3$}; 
    \node[blob] (v4) at (45*3:2) {};\node at (45*3:2.5) {$v_4$}; 
    \node[blob] (v5) at (45*4:2) {};\node at (45*4:2.5) {$v_5$}; 
    \node[blob] (v6) at (45*5:2) {};\node at (45*5:2.5) {$v_6$}; 
    \node[blob] (v7) at (45*6:2) {};\node at (45*6:2.5) {$v_7$}; 
    \node[blob] (v8) at (45*7:2) {};\node at (45*7:2.5) {$v_8$}; 

    \node at (1.2, 0) {$P_2$};
    \node at (0.3, 0.02) {$P_3$};
    \node at (-1.1, 0) {$P_4$};
    
    \draw[myptr] (v1) to (v2);
    \draw (v2) -- (v3);
    \draw[myptr] (v6) to (v5);
    \draw (v6) -- (v7);
    \draw (v8) -- (v1);
    \draw[color=red, thick, myptr] (v3) -- (v4);
    \draw[color=red, thick] (v4) -- (v5);
    \draw[color=red, thick, myptr] (v8) to (v7);

    \draw[dashed, myptr] (v2) to (v8);
    \draw[dashed, myptr] (v7) to (v3);
    \draw[dashed, myptr] (v4) to (v6);
  \end{tikzpicture}
  
  \caption{A Hamilton path with an even number of edges from $R$ (represented in red) in a parity-switcher.} 
  \label{fig:switcher}  
\end{figure}

\paragraph{Proof strategy.} We use the following recipe for showing $\mathcal{C}_n(G) = \mathcal{C}(G)$:
\begin{enumerate}[(S1)]
  \item \label{strategy:E_r} Suppose $\mathcal{C}_n(G) \neq \mathcal{C}(G)$, and take $R \subseteq G$ to be a subgraph given by Lemma \ref{lemma:E_r}.
  \item \label{strategy:find_switcher} Find a (small) $R$-parity-switcher $W$:
    \begin{enumerate}[(S2.a)] \label{strategy:parity_switcher}
      \item \label{strategy:cycle} Find an even cycle $C = (v_1, \ldots, v_{2k})$ with an odd number of edges in $R$.
      \item \label{strategy:paths} Find edge-disjoint (short) paths $P_i$ between $v_i$ and $v_{2k - i + 2}$, for each $2 \le i \le k$.
    \end{enumerate}    
  \item \label{strategy:ham} Find a Hamilton path $H'$ from $v_1$ to $v_{k+1}$ in $G \setminus (V(W) \setminus \{v_1, v_{k+1}\})$.
    \item \label{strategy:parity} If $H'$ contains an odd (even) number of edges in $R$, then choose a Hamilton path $H''$ in $W$ from $v_1$ to $v_{k+1}$ with an even (odd) number of edges in $R$. 
    \item \label{strategy:finish} Conclude that the concatenation of $H'$ and $H''$ gives a Hamilton cycle $H \subseteq G$ with an odd number of edges in $R$, contradicting \ref{prop:counter2}.
\end{enumerate}

A reader familiar with the concept of \emph{absorbers} will notice their similarity to parity-switchers. Indeed, the main goal of the parity-switcher is to reduce the problem to the one where we just aim to find a Hamilton cycle, without any further restrictions (technically, the goal of absorbers is to reduce a problem to the one about almost-spanning structures; morally, its goal is to change the problem in a way which provides us more flexibility).

To implement the strategy, we only need to take care of steps \ref{strategy:find_switcher} and \ref{strategy:ham}. In particular, as only \ref{strategy:cycle} is concerned with $R$, this is the most difficult step (that being said, \ref{strategy:paths} and \ref{strategy:ham} are not simple, however there is already an existing machinery which we make use of). The following lemma takes care of Step \ref{strategy:cycle}, and is the crux of our proof.

\begin{lemma}[Cycle Lemma] \label{lemma:cycle_lemma}
    Let $G$ be a graph and $R \subseteq G$ such that the following holds for some $\ell \in \mathbb{N}$:
    \begin{enumerate}[(L1)]
        \item  \label{cycle:diam_R} For any $S \subseteq V(G)$ of size $|S| \le 2\ell$ and any two vertices $x, y \in V(G) \setminus S$, there is a path between $x$ and $y$ in $R \setminus S$ of length at most $\ell-1$,
        \item \label{cycle:nonG} $R \neq G$ and $R \neq G[A, B]$ for every partition $V(G) = A \cup B$.
    \end{enumerate}
    Then there exists an even cycle $C \subseteq G$ of length $|C| \le 2\ell$ that contains an odd number of edges from $R$.
\end{lemma}
\begin{proof}
    As $R \neq G$, there exists an edge $uw \in G \setminus R$. For $i \in\mathbb{N}$, denote with $\Nod^i(u)$ the set of vertices in $V(G) \setminus \{w\}$ reachable by an odd path (that is, a path with an odd number of edges) of length at most $i$ from $u$ in $R \setminus \{w\}$. Define $\Nev^i(u)$ analogously for even paths, and note that $u \in \Nev^i(u)$. We distinguish two cases.

    \paragraph{Case 1: $\Nod^{\ell}(u) \cap \Nev^{\ell}(u) \neq \emptyset$.} Choose any $z \in \Nod^{\ell}(u) \cap \Nev^{\ell}(u) $, and let $P_o$ and $P_e$ be a shortest odd and even path from $u$ to $z$ in $R \setminus \{w\}$. By assumption, the set $S = V(P_o) \cup V(P_e)\setminus\{z\}$ is of size $|S| \le 2\ell$ Therefore, by~\ref{cycle:diam_R}, we can find a path $P$ of length at most $\ell-1$ from $z$ to $w$ in $R \setminus S$. If this path has even length then set $P' = P_o$, and otherwise $P' = P_e$. In any case, we get an even cycle of length at most $2\ell$ with an odd number of edges from $R$:
    $$
        w \to u \stackrel{P'}{\to} z \stackrel{P}{\to} w
    $$
    
    \paragraph{Case 2: $\Nod^{\ell}(u) \cap \Nev^{\ell}(u) = \emptyset$.} In this case, both $\Nod^{\ell-1}(u)$ and $\Nev^{\ell-1}(u)$ are independent sets in $R$.
    If there is an edge $wz \in R$ for some $z \in \Nev^{\ell-1}(u)$, then a shortest even path in $R$ from $u$ to $z$ together with $z \to w \to u$ gives the desired cycle. Therefore, we can assume all the edges between $w$ and $\Nev^{\ell-1}(u)$ are not in $R$. Note that by \ref{cycle:diam_R}, we have that $\Nev^{\ell-1}(u) \cup \Nod^{\ell-1}(u) \cup \{w\} = V(G)$. This implies $R \subset G[A, B]$ for $A = \Nev^{\ell-1}(u) \cup \{w\}$ and $B = \Nod^{\ell-1}(u)$, thus $R$ is bipartite. Since $R \neq G[A, B]$, there exist vertices $x \in A$ and $y \in B$ such that $xy \in G \setminus R$. Let $P$ be a shortest path in $R$ from $x$ to $y$. Then by \ref{cycle:diam_R} $P$ is of length at most $\ell-1$, and as $R$ is bipartite, it is necessarily odd. Therefore, $P$ and the edge $xy$ close the desired cycle. 
\end{proof}

To demonstrate the overall strategy, we give a short proof of Theorem \ref{thm:dirac}.

\begin{proof}[Proof of Theorem \ref{thm:dirac}]

Suppose $\mathcal{C}_n(G) \neq \mathcal{C}(G)$. Let $R \subseteq G$ be a subgraph  given by Lemma \ref{lemma:E_r}, which we can apply since $G$ is Hamiltonian by Dirac's theorem~\cite{dirac1952some}. Note that we only need to implement steps \ref{strategy:parity_switcher} and \ref{strategy:ham}.
  
\paragraph{Step \ref{strategy:cycle}.} Let $\ell = 10$. We use Lemma \ref{lemma:cycle_lemma} to obtain a cycle $C = (v_1, \ldots, v_{2k})$ in $G$, of even length at most $2\ell$, which contains an odd number of edges from $R$. Condition \ref{cycle:nonG} follows from \ref{prop:counter1} and the second part of \ref{prop:counter3}. To see that \ref{cycle:diam_R} holds, first note that \ref{prop:counter3} implies $R$ has minimum degree at least $n/4 + C/2$ and it is connected. Choosing $C > 4\ell$, we have that for any $S \subseteq V(G)$ of size $|S| \le 2 \ell < C/2$, the minimum degree of $R \setminus S$ is strictly larger than $n/4$. If a shortest path between some $x$ and $y$ in $R\setminus S$ is of length at least $10$, then the neighbourhoods of the first, fourth, seventh, and tenth vertex on such a path are disjoint -- which cannot be given the minimum degree of $R \setminus S$.

\paragraph{Step \ref{strategy:paths}.} Next, we need to connect $v_i$ to $v_{2k - i + 2}$ for every $2 \le i \le k$. This can be done by sequentially applying the fact that as long as the set $S$ consisting of the vertices on the cycle and all the obtained paths so far is of size less than $C$, any two vertices $x,y \in V(G)$ are at distance at most two in $G \setminus (S\setminus \{x,y\})$. As the set $S$ is always of size at most $3\ell$ ($2\ell$ vertices on the cycle, and at most one additional vertex for each pair of vertices we are aiming to connect), this is indeed the case for $C > 3\ell$. This concludes the construction of an $R$-parity-switcher $W$. 

\paragraph{Step \ref{strategy:ham}.} Finally, $G' := G \setminus (V(W) \setminus \{v_1, v_{k+1}\})$ has minimum degree $\delta(G') > n/2$, thus by Dirac's theorem it contains a Hamilton path between any two prescribed vertices.
\end{proof}

\section{Properties of pseudorandom graphs} \label{sec:prelim}

We make extensive use of the fact that pseudorandom graphs are good \emph{expanders}.

\begin{definition}
    We say that a graph $G$ is \emph{$(s, d)$-expanding}, for $s, d \in \mathbb{N}$, if for every subset $S \subseteq V(G)$ of size $|S| \le s$ we have
    $$
        |N(S)| \ge d |S|.
    $$
\end{definition}

\begin{lemma}\label{lemma:jumbled_to_expand}
   There exists $\eps > 0$ such that the following holds for any $n$ and $0<p<1$. Suppose $G$ is an $n$-vertex $(p,\beta)$-jumbled graph with $\delta(G) \geq (1-\varepsilon)np$, where $\beta \leq \eps n p / d$ for some $d \in \mathbb{N}$. Then $G$ is $(n/(3d),d)$-expanding.
\end{lemma}
\begin{proof}
    Suppose that $G$ is not $(n/(3d),d)$-expanding. Let $S\subseteq V(G)$ be a set of vertices of size $|S|\leq n/(3d)$ such that $N(S)<d|S|$. 
    Let $X_1,X_2,...,X_\ell$ be a partition of $N(S)$ such that $|X_1|=|X_2|=...=|X_{\ell-1}|=|S|$ and $|X_\ell| \leq |S|$. Note that $\ell \leq d$. Let $F\subseteq G$ be the set of edges with at least one endpoint in $S$. By the assumptions on $G$, we have:
    $$
    |F|\geq (1-\varepsilon)|S|np-e_G(S)\geq (1-\varepsilon)|S|np-p|S|^2/2-\eps np|S|/d\geq\Big(1-\frac{1}{6d}-\varepsilon\Big)|S|np-\frac{\eps np|S|}{d}.
    $$
    On the other hand, since $G$ is $(p,\beta)$-jumbled, for $i\in[\ell]$ we get:
    $$
    e_G(S\cup X_i)\leq p\binom{2|S|}{2}+2\eps np|S|/d\leq 2|S|^2p+2\eps np|S|/d.
    $$
    Therefore, 
    $$|F|\leq 2d|S|^2p+2\eps np|S|.$$
    Combining the above two observations, we get
    $$
    \Big(1-\frac{1}{6d}-\varepsilon\Big)|S|np-\frac{\eps np|S|}{d}\leq|F|\leq 2d|S|^2p+2\eps np|S|
    $$
    and so 
    $$
    \frac{1-1/(6d)-4\varepsilon}{2d}n\leq |S|,
    $$
    contradicting our assumption $|S| \leq n/(3d)$.
\end{proof}

\begin{lemma}\label{lemma:jumbled_to_cuts}
There exists $\eps > 0$ such that the following holds for any $n$ and $0<p<1$. Let $\beta\leq \eps np$, and suppose $G$ is an $n$-vertex $(p,\beta)$-jumbled graph with $\delta(G) \geq (1-\varepsilon)np$. Then for any disjoint $A,B \subseteq V$ with $A \cup B = V$, we have
$$ e_G(A,B) \geq (1-6\varepsilon)p|A||B|.$$
\end{lemma}
\begin{proof}
Suppose w.l.o.g. $|A| \leq |B|$, so $|B| \geq n/2$. Then, since $G$ is $(p,\beta)$-jumbled,
$$ e_G(A) \leq p \binom{|A|}{2} + \beta |A|,$$
so the minimum degree of $G$ implies
$$ e_G(A, B) \geq |A| (1-\varepsilon)np - 2e_G(A) \geq |A| (1-\varepsilon)np - |A|^2p - 2|A| \eps np \geq |A|(n-|A|)p-3\varepsilon|A|np 
$$ 
$$ \geq |A|(n - |A| - 3\eps n)p\geq |A| |B|\Big(1 - \frac{3\eps n}{|B|}\Big)p \geq (1-6\varepsilon)|A| |B|p. $$
\end{proof}

\begin{lemma} \label{lemma:jumbled_to_sparse}
There exists $\eps > 0$ such that the following holds for any $n$ and $0<p<1$. Let $\beta\leq \eps np$.
Let $G$ be an $n$-vertex $(p,\beta)$-jumbled graph. Then for any disjoint $A,B \subseteq V$, we have
$$ e_G(A,B) \leq |A||B|p + 2\beta(|A| + |B|).$$
\end{lemma}
\begin{proof}
By $G$ being $(p,\beta)$-jumbled, we have
$$ e_G(A \cup B) \leq p \binom{|A|+|B|}{2} + \beta(|A| + |B|) $$
$$ e_G(A) \geq p \binom{|A|}{2} - \beta |A| $$
$$ e_G(B) \geq p \binom{|B|}{2} - \beta |B|. $$
Combining these, we get
$$ e_G(A,B) \leq p \binom{|A|+|B|}{2} - p \binom{|A|}{2} - p \binom{|B|}{2} + 2\beta (|A| + |B|) = p|A||B|+ 2\beta (|A| + |B|). $$
\end{proof}

Finally, the last lemma shows that large subgraphs of $(p, \beta)$-graphs are robust with respect to having a small diameter, which allows us to apply Lemma~\ref{lemma:cycle_lemma}. 

\begin{lemma} \label{lemma:sparse_diam}
    There exist $\eps > 0$ and $C > 1$ such that the following holds for sufficiently large $n$. Let $G$ be an $n$-vertex $(p, \beta)$-jumbled graph with $\delta(G) \ge (1-\eps)np$, where $p \ge C \log n / n$ and $\beta \leq \eps np$. Suppose $R \subseteq G$ is such that for every partition $A \cup B = V(G)$ we have $e_R(A, B) \ge e_G(A, B)/2$. Then, for any $S \subseteq V(G)$ of size $|S| \leq 2 \log n$, $R' = R \setminus S$ has the property that between any two vertices $v, w \in V(G) \setminus S$ there is a path of length at most $\log n-1$ in $R'$.
\end{lemma}
\begin{proof}
Let $S$ be as in the statement of the lemma. We denote $V(G)$ by $V$, and we start by showing the following helper expansion claim.
\begin{claim}
\label{claim:expansion_of_set}
If $X \subseteq V$ has size at most $n/2$, then we have
$$ |N_{R}(X,V\setminus (X \cup S))| \geq \min\{10|X|, n/7\}.$$
\end{claim}
\begin{proof}
Let $Y:= N_{R}(X, V\setminus X)$. We first suppose $|X| \leq 2\log{n}$. For any $x \in X$ we have $\deg_R(x) \geq \deg_G(x)/2 \geq np/3$, thus for sufficiently large $C$ we have
$$ |N_R(X, V\setminus (X \cup S))| \geq \deg_R(x, V \setminus (X \cup S)) \geq C \log{n}/3 - |X| - |S| \geq 20\log{n} \geq 10|X|.$$
Next, consider the case $2\log{n} \leq |X| \leq n/200$. Suppose, towards a contradiction, that $|Y| < 20|X|$. Then, by Lemma~\ref{lemma:jumbled_to_cuts} and Lemma~\ref{lemma:jumbled_to_sparse}, we have
$$  |X| np/6  \leq e_R(X,V\setminus X) = e_R(X,Y) \leq 20|X|^2 p + 42 \beta |X| \leq 20|X| np/200 + 42 \eps np |X| <  |X| np/6,$$
where we used $|V\setminus X| \geq n/2$ and $e_R(X,V\setminus X) \geq e_G(X,V\setminus X)/2$ for the first inequality. This is a contradiction, thus $|Y| \geq 20|X|$. Therefore, in this case we have
$$|N_R(X, V\setminus (X \cup S))| = |Y \setminus S| \geq 20|X| - 2\log{n} \geq 10|X|.$$
Finally, if $|X| \geq n/200$, then, again by Lemma~\ref{lemma:jumbled_to_cuts} and Lemma~\ref{lemma:jumbled_to_sparse},
$$  |X| np/6 \leq e_R(X, Y) \leq |X| |Y|p + 2\beta (|X| + |Y|) \leq |X||Y| p + 2\eps n^2 p \leq |X||Y| p + 400\eps n|X| p,$$
implying that $|Y| \geq (1 / 6 - 400\eps)n$, therefore $|N_R(X, V\setminus (X \cup S)| = |Y \setminus S| \geq n/7$.
\end{proof}
For a set $A \subseteq V$, a vertex $x \in A$ and an integer $\ell\geq 1$, we denote by $N^{\ell}_R(x,A)$ the set of vertices in $A$ reachable from $x$ in $R[A]$ via a path of length at most $\ell$. Note that $N^0_R(x, A) = \{x\}$.
\begin{claim}
\label{claim:expansion_to_half}
    For any $x \in V\setminus S$  and $\ell = \log_{10}{n} + 7$, we have
    $$ |N^{\ell}_R(x, V \setminus S)| > n/2.$$
\end{claim}
\begin{proof}
Starting with $N^0_R(x, V \setminus S)$, we iteratively apply Claim~\ref{claim:expansion_of_set} to $X:= N^i_R(x, V\setminus S)$ to show that
$$ |N^{i+1}_R(x, V\setminus S)| \geq |N^i_R(x, V\setminus S)|  + \min\{10^{i+1}, n/7 \}.$$
\end{proof}

By Claim~\ref{claim:expansion_to_half}, for every $v,w \in V\setminus S$ with $v \neq w$, we have that $N^{\ell}_R(v,V \setminus S) \cap N^{\ell}_R(w,V \setminus S) \neq \varnothing$, implying that there is a path between $v$ and $w$ in $R' = R \setminus S$ of length at most $$2\ell \leq 2\log_{10}{n} + 14 \leq 2\frac{\log{n}}{\log{10}} + 14 \leq 0.87 \log{n} + 14 \leq \log{n} - 1.$$
\end{proof}

\subsection{Connecting vertices} 

The previous section summarises important properties of $(\beta, p)$-jumbled graphs with sufficiently large minimum degree. Here we state some known results which rely on these properties, and which we use to implement steps \ref{strategy:paths} and \ref{strategy:ham}.

\begin{theorem}[\cite{hefetz09ham}] \label{thm:ham_connected}
Let $G$ be a graph with $n$ vertices, and suppose the following two properties hold:
\begin{itemize}
    \item $G$ is $\big(\frac{4n}{\log{n}}, \log\log n/4\big)$-expanding, and
    \item between every two disjoint subsets of vertices $A, B \subseteq V(G)$ such that $|A|, |B| \ge \frac{n \log\log n }{5000\log n }$, there exists an edge.
\end{itemize}
Then, for $n$ sufficiently large and for any two vertices $x, y \in V(G)$, there exists a Hamilton path in $G$ between $x$ and $y$.
\end{theorem}
To find a switcher in $G$ but still be able to apply the above theorem, we have to connect some vertices with paths while maintaining that the remaining graph is an expander. In the case of Theorem \ref{thm:dirac} we got that, in some sense, for free because of the large minimum degree. The following definition and two statements provide us with the framework to achieve that in the sparse setting.

\begin{definition}
    Let $G$ be a graph and $s, D \in \mathbb{N}$. Given a graph $F$ with maximum degree at most $D$, we say that an embedding $\phi \colon F \hookrightarrow G$
    is \emph{$(s, D)$-good} if for every $X \subseteq V(G)$, of size $|X| \le s$, we have
    $$
        |N_G(X) \setminus \phi(F)| \ge |\phi(F) \cap X| + \sum_{v \in X} \left[ D - \deg_F(\phi^{-1}(v)) \right].
    $$
    (Note: We slightly abuse the notation by letting $\deg_F(\varnothing) = 0$).
\end{definition}

\begin{theorem}[\cite{draganic22rolling}] \label{thm:fp_tree}
    Let $F$ be a graph with maximum degree at most $D$ and fewer than $s/2+1$ vertices, for some $D, s \in \mathbb{N}$. Suppose we are 
    given an $(s, D+2)$-expanding graph 
    $G$ and an $(s, D)$-good embedding 
    $\phi \colon F \hookrightarrow G$. Then for every graph $F' \supseteq F$ with maximum degree at most $D$ and at most $s/2+1$ vertices, which can be obtained from $F$ by successively adding a new vertex of degree 1, there exists an $(s, D)$-good embedding $\phi' \colon F' \hookrightarrow G$ which extends $\phi$.
\end{theorem}

\begin{lemma}[\cite{draganic22rolling}] \label{lemma:fp_rem}
    Let $F$ be a graph with at most $s$ vertices and maximum degree at most $D$, for some $s, D \in \mathbb{N}$. Suppose we are given a graph $G$ and an $(s, D)$-good
    embedding $\phi \colon F \hookrightarrow G$. Then for every graph $F'$ obtained from $F$ by successively removing a vertex of degree 1, the restriction $\phi'$ of $\phi$ to $F'$ is also $(s, D)$-good.
\end{lemma}

\section{Hamilton space in sparse pseudorandom graphs}  \label{sec:proof_technical}

Throughout the section we use $V$ to denote the vertex set of $G$. We follow the recipe given in Section \ref{sec:outline}.

\begin{proof}[Proof of Theorem \ref{thm:main_technical}]
Note first that $G$ is Hamiltonian. That is because it is $(\frac{n}{3\log\log{n}}, \log\log{n})$-expanding by Lemma~\ref{lemma:jumbled_to_expand}, and so it satisfies the conditions of Theorem~\ref{thm:ham_connected}. Thus for any edge $uv \in E(G)$, there is a Hamilton path between $u$ and $v$ in $G$, closing a Hamilton cycle.

Suppose $\mathcal{C}_n(G) \neq \mathcal{C}(G)$, and let $R \subseteq G$ be a subgraph given by Lemma \ref{lemma:E_r}. We now show how to implement steps \ref{strategy:find_switcher} and \ref{strategy:ham}, which suffices for the contradiction. We refer the reader to Section \ref{sec:outline} for details. We use the following values:
$$
    \ell = \log n, \quad d = \log \log n/2, \quad \text{ and } \quad s = n / (6d).
$$

\paragraph{Step \ref{strategy:cycle}.} By applying Lemma \ref{lemma:cycle_lemma} with $\ell$, which we can do due to Lemma \ref{lemma:sparse_diam}, we obtain an even cycle $C = (v_1, \ldots, v_{2k})$ in $G$ of length at most $2\ell$ that contains an odd number of edges in $R$.

\paragraph{Step \ref{strategy:paths}.} We now find paths connecting $v_i$ to $v_{2k - i + 2}$, for $2 \le i \le k$, which together with $C$ form an $R$-parity-switcher. Let $F$ be a graph with vertex set $\{u_1, \ldots, u_{2k}\}$ ($F$ does not live in $G$), and let $\phi \colon F \hookrightarrow G$ be defined as $\phi(u_i) = v_i$. Note that $|\phi(F)| = 2k \leq 2\ell$. Recall that by Lemma \ref{lemma:jumbled_to_expand}, $G$ is $(n/(6d),2d)$-expanding. Together with $\delta(G)\geq np/2$, for every $X\subseteq V(G)$ with $|X| \leq n / (6d)$ we have
$$
    |N(X)|\geq \max\{np/2, 2d |X| \}.
$$
This implies $\phi$ is an $(s, d)$-good embedding: Consider some $X \subseteq V(G)$ of size $|X| \le n / (6d)$.
\begin{itemize}
    \item If $|X| \leq 4 \ell / d$, then
    $$
        |N_G(X) \setminus \phi(F)| \geq np/2 - |\phi(F)| \geq 6 \ell \geq |\phi(F) \cap X| + d |X|.
    $$
    
    \item If $4 \ell / d \le |X| \le n / (6d)$, then
    $$
        |N_G(X) \setminus \phi(F)| \geq 2d|X| - |\phi(F)| \geq 2d|X| - 2 \ell \ge 2 \ell + d |X| \ge |\phi(F) \cap X| + d|X|.
    $$
\end{itemize} 

Let $B_n$ be a binary tree with $s/10$ vertices and of depth at most $\log n$. We connect $v_i$ and $v_{2k-i+2}$ for $i \in \{2, \ldots, k\}$, sequentially, with paths of length at most $\log n$, as follows: 
\begin{itemize}
    \item Let $F'$ be a graph obtained from $F$ by attaching disjoint copies of $B_n$ to $u_i$ and $u_{2k-i+2}$. Denote these two copies of $B_n$ in $F'$ by $B_n^a$ and $B_n^b$.
    \item Note that $\Delta(F') \leq 3 \leq d$ and $|V(F')| \leq s/5 + 2\ell \cdot \log{n} \le s/2$. We may apply Theorem \ref{thm:fp_tree} to extend $\phi$ to an $(s, d)$-good embedding $\phi'$ of $F'$.
    \item Let $a \in \phi'(B_n^a)$ and $b \in \phi'(B_n^b)$ be such that $ab$ is an edge in $G$, which exists since 
    $$
        \frac{s}{10} = \frac{n}{30\log\log n} \geq \frac{\eps n \log{\log{n}}}{ \log{n}}.
    $$
    Let $P_{i}$ denote the path in $G$ from $v_i$ to $v_{2k-i+2}$ going through $\phi'(B_n^a)$ and $\phi'(B_n^b)$ using $ab$.
    \item Let $F''$ be obtained from $F'$ by removing all the vertices in $B_n^a$ and $B_n^b$ which are not on $\phi'^{-1}(P_i)$.
    \item Set $F := F''$ and $\phi$ to be the restriction of $\phi'$ to $F''$, which remains an $(s,d)$-good embedding by Lemma~\ref{lemma:fp_rem}.
\end{itemize}
At the end of this process, we have that every pair of vertices $v_i, v_{2k-i+2}$ is connected by a path $P_i$ in $G$, thus $\phi(F)\cup C$ is an $R$-parity-switcher which we denote by $W \subseteq G$. Crucially, the embedding $\phi$ is $(s,  d)$-good. We use this immediately in the next step. 

\paragraph{Step \ref{strategy:ham}.} Since $\Delta(F)\leq 3$, we get that
$$
    G' = G \setminus (\phi(F) \setminus \{v_1, v_{k+1}\})
$$
is $(s, d/2)$-expanding. To see this, consider any set $X\subseteq V(G')$ with $|X|\leq s$. Then
$$
    |N_{G'}(X)| \geq |N_G(X) \setminus \phi(F)|\geq (d-3)|X|\geq  d|X|/2,
$$
where the second inequality follows from the fact that $\phi$ is $(s,d)$-good. We can now apply Theorem \ref{thm:ham_connected} to obtain a Hamilton path $P$ in $G'$ between $v_1$ and $v_{k+1}$, which finishes the proof.
\end{proof}

\section{Concluding remarks} \label{sec:concluding}

\begin{itemize}
    \item Heinig \cite{heinighamiltonspace} showed that if $\mathcal{C}_n(G_{n,p}) = \mathcal{C}(G_{n,p})$ and $\Gnp$ is not a forest, then necessarily $\delta(G_{n,p}) \ge 3$. This prompts the following \emph{hitting time} problem, which would further refine Corollary \ref{cor:gnp}:

    \begin{problem}
        Suppose $n$ is odd, and consider the random graph process $\{G_m\}_{m \in \binom{n}{2}}$. Is it true that, with high probability, $\delta(G_m) \ge 3$ implies $\mathcal{C}_n(G_m) = \mathcal{C}(G_m)$?
    \end{problem}

    \item We reiterate the question asked by Heinig~\cite{heinig2014prisms}:

    \begin{problem} \label{prob:dirac}
        Suppose $n$ is odd and sufficiently large, and let $G$ be a graph with $\delta(G) > n/2$. Is it true that $\mathcal{C}_n(G) = \mathcal{C}(G)$?
    \end{problem}

    Theorem \ref{thm:dirac} answers this if $\delta(G)$ is just a bit larger, namely $\delta(G) \ge n/2 + C$ (taking $C = 41$ suffices). It would be interesting to see if the strategy outlined in Section \ref{sec:outline}, together with a stability-type case analysis akin to the one in \cite{krivelevich14robust}, for example, can be used to resolve Problem \ref{prob:dirac}. We leave this for future work.
\end{itemize}

\bibliographystyle{abbrv}
\bibliography{bibliography.bib}
\end{document}